\documentclass[a4paper,12pt]{article}
\usepackage{amsmath,amssymb,amsfonts,latexsym, amsthm}
\usepackage[active]{srcltx}

\newtheorem{teor}{\bf Theorem}[section]

\newtheorem{defi}[teor]{\bf Definition}
\newtheorem{lemma}[teor]{\bf Lemma}

\newtheorem{rmk}{\bf Remark}

\newtheorem{defn}{Definition}[section]

\newcommand{\Rn}{{\mathbb R}^{N}}
\newcommand{\Hn}{{\mathbb H}^{N}}
\newcommand{\Bn}{{\mathbb B}^{N}}

\newcommand{\fa} {\forall}

\newcommand{\al} {\alpha}
\newcommand{\ba} {\beta}
\newcommand{\de} {\delta}
\newcommand{\ga} {\gamma}

\newcommand{\Ome} {\Omega}

\newcommand{\De} {\Delta}
\newcommand{\la} {\lambda}
\newcommand{\si} {\sigma}
\newcommand{\Si} {\Sigma}
\newcommand{\no} {\nonumber}

\newcommand{\na} {\nabla}
\newcommand{\va} {\varphi}
\newcommand{\var} {\varepsilon}

\newcommand{\spa}{\vspace{.2in}}

\newcommand{\lan}{\langle}
\newcommand{\ran}{\rangle}

\newcommand{\R}{\mathbb{R}}
\newcommand{\N}{\mathbb{N}}

\newcommand{\Hunb}{H^{1}(\Bn)}

\newcommand{\Ibn}{\int_{\Bn}}

\newcommand{\Il}{I_{\la}}

\newcommand{\wn}{w_{n}}
\newcommand{\um}{u_{m}}
\newcommand{\un}{u_{n}}
\newcommand{\vm}{v_{m}}
\newcommand{\vn}{v_{n}}
\newcommand{\bn}{b_{n}}

\newcommand{\deb}{\rightharpoonup}

\newcommand{\nab}{\nabla_{\Bn}}
\def\H{{I\!\!H}}

\title{Poincar\'{e} Sobolev equations in 
\\
 the Hyperbolic space}

\author{Mousomi Bhakta and K. Sandeep\footnote{TIFR Centre for Applicable Mathematics, Post Bag No. 6503
 Sharadanagar,Chikkabommasandra, Bangalore 560065. Email: mousomi@math.tifrbng.res.in; sandeep@math.tifrbng.res.in}}

\date{}
\begin{document}

\maketitle
\begin{abstract}
We study the a priori estimates,existence/nonexistence of radial 
sign changing solution, and the Palais-Smale characterisation of the problem $-\De_{\Bn}u - \la u = |u|^{p-1}u , u\in H^1(\Bn)$ in the hyperbolic space $\Bn$ where $1<p\leq\frac{N+2}{N-2}$. We will also prove the existence of sign changing solution to the Hardy-Sobolev-Mazya equation and the critical Grushin problem. 
\end{abstract}

\section{Introduction}

In this article we will study compactness properties and the existence/non-existence of sign changing solutions of the problem
\begin{equation}\label{problem}
 -\De_{\Bn}u - \la u = |u|^{p-1}u , u\in H^1(\Bn)
\end{equation}
where $1<p\le \frac{N+2}{N-2}$, $\la<(\frac{N-1}{2})^2$ and $H^1(\Bn)$ denotes the Sobolev space on the disc model of the Hyperbolic space $\Bn$ and $\De_{\Bn}$ denotes the Laplace Beltrami operator on $\Bn$.\\
Apart from its own mathematical interest, the equation \eqref{problem} is closely related to the 
study of Hardy-Sobolev-Mazya type equations and Grushin operators under partial symmetry of their 
solutions(See \cite{CFMS1},\cite{CFMS2},\cite{MS}).\\
First consider the Hardy Sobolev Mazya type equations
\begin{equation}\label{P}
- \Delta u \;-\; \eta \frac{u}{|y|^2}= \frac{|u|^{p_t-1}u}{|y|^t} \;\mbox{ in } \;\R^n,\;  u \in  D^{1,2}(\R^n) 
\end{equation}
where $\R^n =\R^k\times \R^{n-k}, 2\le k< n, 0\le \eta < \frac{(k-2)^2}{4} $ when $k>2, \eta =0$ when 
$k=2, 0\le t< 2 $ and $p_t =\frac{n+2-2t}{n-2}$. A point $x\in \R^n$ is denoted as 
$x=(y,z)\in \R^k\times \R^{n-k}$ .\\
One can see that $ u \in  D^{1,2}(\R^n)$ is a cylindrically symmetric solutions of \eqref{P} 
(i.e., $u(x) = \tilde u(|y|,z)$) iff $v = w\circ M $ solves \eqref{problem} with dimension $N = n-k+1,p=p_t$ and $ \la = \eta + \frac{(n-k)^2-(k-2)^2}{4}$  where 
$w(r,z)= r^{\frac{n-2}{2}}\tilde u(r,z)$ for $(r,z) \in (0,\infty)\times \R^{n-k}$ and 
$M : B^{n-k+1} \rightarrow (0,\infty)\times \R^{n-k}$ is the standard isometry (see (\ref{H-B}) ) 
between the $\Bn$ and the upper half space model of the hyperbolic space (see \cite{CFMS1}, 
\cite{MS} for details.) Note that when $k=2$ and $\eta=0$, then we have $\la = (\frac{N-1}{2})^2$ in \eqref{problem}. But in this case the assumption $u\in H^1_0(\Bn)$ has to be replaced by $u\in \cal{H}$ wich is the completion of $C_c^\infty(\Bn)$ with the norm $\left( \int\limits_{\Bn} \left[ |\nabla_{\Bn}  u|^2  - {\frac{(n-1)^2 }{4}} u^2 \right]  dV_{\Bn}\right)^{\frac12}$,see  \cite{MS} for details.\\\\
The critical Grushin-type equations are given by
\begin{equation}\label{GRU}
\Delta_y \va +(1+\alpha)^2|y|^{2\alpha}\De_z \va + \va^{\frac{Q+2}{Q-2}} = 0 \qquad  \quad (y,z) \in\R^k \times 
\R^h\\
\end{equation} 
where $ \alpha >0,  \:  Q=k+h(1+\alpha) \:  ,  k, h \geq 1$ . The equation (\ref{GRU}) can be considered as the Euler-Lagrange equation satisfied by the extremals of the  weighted 
Sobolev inequality 
\begin{equation}\label{Sobgrucil}
\tilde S \left(\int\limits_{\R^N}|u|^{\frac{2Q}{Q-2}} dydz \right)^{\frac{Q-2}{Q}} \leq
 \int\limits_{\R^N} \left(|\nabla_y u|^2 + (\alpha + 1)^2 |y|^{2\alpha} |\nabla_z u|^2\right) dy dz 
\end{equation} 
Connections between Grushin operators and hyperbolic geometry were observed by Beckner \cite{B}.
 As before $\va$ is a cylindrically symmetric solution of (\ref{GRU}) with 
$ \||u\||^2 := \int\limits_{\R^N} \left(|\nabla_y u|^2 + (\alpha + 1)^2 |y|^{2\alpha} |\nabla_z u|^2\right) dy dz < \infty$,
iff $u= \Phi \circ M$ solves \eqref{problem} with 
$$N= h+1, \qquad \la = \frac 14 \left[ h^2-(\frac{k-2}{\al +1})^2\right] , \qquad p=\frac{Q+2}{Q-2} < 2^\ast-1 $$
where  
 $\Phi(r,z)=r^{\frac{Q-2}{2(1 +\al)}}\, \va(r^{\frac{1}{1+\al}},z), \:  r=|y|$, 
Here again when $\la = (\frac{N-1}{2})^2$ in \eqref{problem} the assumption $u\in H^1_0(\Bn)$ has to be replaced by $u\in \cal{H}$.\\\\
Positive solutions of \eqref{problem} has been extensively studied in \cite{MS}. In fact it was shown in \cite{MS} that \eqref{problem} has a positive solution iff either $1<p<\frac{N+2}{N-2} $ and $\la < \frac{(N-1)^2}{4}$ or $p=\frac{N+2}{N-2} ,\frac{N(N-2)}{4}<\la < \frac{(N-1)^2}{4}$ and $N\ge 4.$ The solutions are also shown to be unique up to isometries (except in $N=2$ where there is a restriction on $p$). \\
In this article we focus on sign changing solutions of \eqref{problem}. The subcritical case is quite different from the critical case where the lack of compactness of the problem comes in to picture. We will present this compactness analysis in Section 3, Theorem \ref{cpt} and Theorem \ref{PS}. In section 4, we will some appriori estimates on the solution. In the fifth section we will prove our main existence results Theorem \ref{K}, Theorem \ref{HSMSC}, Theorem \ref{GRUSC} and Theorem \ref{CP}. Some preliminaries about Hyperbolic space are discussed in the appendix.\\

\section{Priliminaries} 

Let $\Bn:=\{x\in\Rn: |x|<1\}$ denotes the unit disc in $\Rn$. The space $\Bn$ endowed 
with the Riemannian metric $g$ given by $g_{ij}=(\frac{2}{1-|x|^2})^2\de_{ij}$ is called
the ball model of the Hyperbolic space.\\\\ We will denote the associated hyperbolic 
volume by $dV_{\Bn}$ and is given by  $dV_{\Bn}=(\frac{2}{1-|x|^2})^N dx$.
The hyperbolic gradient $\na_{\Bn}$ and the hyperbolic Laplacian $\De_{\Bn}$ are given by
$$
 \na_{\Bn}=(\frac{1-|x|^2}{2})^2\na,\ \ \  \De_{\Bn}=(\frac{1-|x|^2}{2})^2\De+(N-2)\frac{1-|x|^2}{2}<x,\na>
$$
Let $H^1(\Bn)$ denotes the Sobolev space on $\Bn$ with the above metric $g$, then we have
$H^1(\Bn) \hookrightarrow L^p(\Bn)$ for $2\le p\le \frac{2N}{N-2}$ when $N\ge 3$ and
$p\ge 2$ when $N=2$. In fact we have the following Poincar\'{e}-Sobolev inequality (See \cite{MS}) :\\
For every $N \geq 3$ and every $p\in (2, \frac{2N}{N-2}] $
there is an optimal constant \\
$S_{N,p,\la }>0$ such that
\begin{equation}\label{PSE}
 S_{N,p,\la} \left(\int\limits_{\Bn}  |u|^{p} dV_{\Bn} \right)^{\frac{2}{p}} \leq 
\int\limits_{\Bn} \left[ |\nabla_{\Bn}  u|^2  - {\frac{(n-1)^2 }{4}} u^2 \right]  dV_{\Bn} \quad 
\end{equation}
for every $ u\in H^1(\Bn).$  Existence of exremals for \eqref{PSE} and their uniqueness has been studied in \cite{MS}. If $N=2$ any $p> 2$ is allowed (See \cite{AT}, \cite{MS1} for a more precise embedding in this case).\\
Thanks to \eqref{PSE} solutions of \eqref{problem} can be characterised as the critical points of the energy functional $\Il$ given by
\begin{equation}\label{IL}
I_{\la}(u)= \frac{1}{2}\Ibn \big[|\na_{\Bn}u|_{\Bn}^2-\la u^2\big]dV_{\Bn}-\frac{1}{p+1}\Ibn |u|^{p+1}dV_{\Bn} 
\end{equation}
{\bf Conformal change of metric.} Let $f: M \rightarrow N$ be a conformal diffeomorphism between two Riemannian manifolds $(M,g)$ and $(N,h)$ of dimension $N\ge 3,$ i.e., $f^\ast h= \phi ^{\frac{4}{N-2}}g$ for some positive function $\phi.$
Consider the equations
\begin{equation}\label{EM}
 -\Delta_gu + \frac{N-2}{4(N-1)}S_g u= |u|^{\frac{4}{N-2}}u \;\;\;{\text on}\;\;\; M
\end{equation}
\begin{equation}\label{EN}
 -\Delta_hv + \frac{N-2}{4(N-1)}S_h v= |v|^{\frac{4}{N-2}}v \;\;\;{\text on}\;\;\; N
\end{equation}
where $\Delta_g,S_g $ and $\Delta_h,S_h $ are the Laplace Beltrami operators and scalar curvatures
 on $M$ and $N$ respectively. Then if $v$ is a solution of \eqref{EN}, then $u=\phi (v\circ f)$
 is a solution of \eqref{EM}. Moreover $\int\limits_M |u |^{\frac{2N}{N-2}}dV_M = \int\limits_N |v|^{\frac{2N}{N-2}} dV_N$ if one of the integral is finite.

As an easy consequence, if $\tau \in I(\Bn)$ the isometry group of $\Bn$ and $u$ any solution of \eqref{problem} then 
$v=u\circ \tau$ is again a solution of \eqref{problem} and $\Il(u) = \Il(v).$ See the appendix for details about the isometry group $ I(\Bn)$.\\

As another consequence, noting that the hyperbolic metric $g=\phi ^{\frac{4}{N-2}}g_e $ where
 $g_e$ is the Euclidean metric on $\Bn$,  
$\phi = \left(\frac{2}{1-|x|^2}\right)^{\frac{N-2}{2}}$ and the scalar curvature of $g$ is $-N(N-1)$ we see that $u$ is a solution of \eqref{problem} with  $p = \frac{N+2}{N-2}$ iff $v= \left(\frac{2}{1-|x|^2}\right)^{\frac{N-2}{2}}u$ solves the Euclidean equation
\begin{equation}\label{EE}
-\De v - \tilde\la \left(\frac{2}{1-|x|^2}\right)^2v =  |v|^{\frac{4}{N-2}}v, v \in H^1_0(\Bn)
 \end{equation}
where $\tilde \la = (\la-\frac{N(N-2)}{4})$.
Let us denote the energy functional corresponding to \eqref{EE} by
\begin{equation}\label{JL}
J_{\la}(v)= \frac{1}{2}\Ibn \big[|\na v|^2-\tilde \la \left(\frac{2}{1-|x|^2}\right)^2v^2\big]dx-\frac{1}{p+1}\Ibn |v|^{p+1}dx 
\end{equation}
Then for any $u \in H^1(\Bn)$  if $\tilde u$ is defined as  $\tilde u =  \left(\frac{2}{1-|x|^2}\right)^{\frac{N-2}{2}}u$ then  $\Il(u)=J_{\la}(v)$.Moreover
$\lan\Il^\prime(u),v\ran=\lan J_{\la}^\prime(\tilde u), \tilde v\ran$ where $\tilde v$ is defined in the same way.

\section{Compactness and non-compactness}

In this section we will study the compactness properties of \eqref{problem}. Let 
$u \in H^1(\Bn)$ and $b_n \in \Bn$ such that $b_n\rightarrow \infty$ and $\tau_n$ be
the Hyperbolic translations(see Appendix) such that  $\tau_n(0)=b_n$. Define $u_n = u\circ\tau_n $,
then $||u_n|| = ||u||$ but $u_n \rightharpoonup 0$ in $ H^1(\Bn)$. This shows that the
embedding $H^1(\Bn) \hookrightarrow L^p(\Bn)$ is not compact for any $2\le p \le \frac{2N}{N-2}$. Hence the problem \eqref{problem} is non compact even in the subcritical case. Below we will show that we can overcome this problem in the subcritical case by restricting to the radial situation. The critical case is more involved, we will show that the noncompactness can occur through two profiles.
\begin{subsection}{The radial case.} Let $H^1_r(\Bn)$ denotes the subspace
$$H^1_r(\Bn)=\{u \in H^1(\Bn): u  \; is\; radial\}$$
Since the hyperbolic sphere with centre $0\in \Bn$ is also a Euclidean sphere with centre $0\in \Bn$ (See the appendix),$H^1_r(\Bn)$ can also be seen as the subspace consisting of Hyperbolic radial functions. 
\begin{teor}\label{cpt}
 The embedding $H^1_r(\Bn)\hookrightarrow L^p(\Bn)$ for $2<p<2^*$ is compact.
\end{teor}
\begin{proof}
 Let $u\in H^1_r(\Bn)$ then $u(x)=u(|x|)$, by denoting the radial function by $u$ itself.
Then 
\begin{displaymath}
\omega_{N-1}\int_0^1 u'(s)^2(\frac{2}{1-s^2})^{N-2}s^{N-1}ds = \Ibn |\na u|^2(\frac{2}{1-|x|^2})^{N-2}dx< \infty. 
\end{displaymath}
where $\omega_{N-1}$ is the surface area of $S^{N-1}.$ Thus for $u\in H^1_r(\Bn)$
\begin{eqnarray*}
 u(|x|) &=& -\int_{|x|}^1 u'(s)ds
\\      
       &\leq& \big(\int_0^1 u'(s)^2(\frac{2}{1-s^2})^{N-2}s^{N-1}ds\big)^\frac{1}{2}\big(\int_{|x|}^1 (\frac{1-s^2}{2})^{N-2}s^{-(N-1)}ds\big)^\frac{1}{2}
\\
       &\leq& \omega_{N-1}^{-\frac{1}{2}}||u||_{\Hunb}(\frac{1-|x|^2}{2})^\frac{N-2}{2}\frac{1}{|x|^\frac{N}{2}}\big(\int_{|x|}^1 sds\big)^\frac{1}{2}
\\
       &\leq& \omega_{N-1}^{-\frac{1}{2}}||u||_{\Hunb}(\frac{1-|x|^2}{2})^\frac{N-1}{2}\frac{1}{|x|^\frac{N}{2}}
\end{eqnarray*}
Let $\{u_m\}$ be a bounded sequence in $H^1_r(\Bn)$. Then upto a subsequence we may assume $u_m\rightharpoonup u$ in $H^1_r(\Bn)$ and pointwise . To complete the proof we need to show now $u_m\to u$ in $L^p(\Bn)$. 
\begin{displaymath}
\Ibn |u_m|^p dV_{\Bn}= \int_{|x|\leq \frac12}|u_m|^p dV_{\Bn}+ \int_{|x|> \frac12}|u_m|^p dV_{\Bn}.
\end{displaymath}
The convergence of 1st integral follows from Relich's compactness theorem. The convergence of
 2nd integral follows from the dominated convergence theorem as in $\{|x|>\frac12\}$ we have the 
estimate $ |u_m(x)|^p \le C\big(\frac{1-|x|^2}{2}\big)^{\frac{N-1}{2}p}$ and  
$$\int_{|x|>\frac12}\big(\frac{1-|x|^2}{2}\big)^{\frac{N-1}{2}p}dV_{\Bn} \le \int_{|x|>\frac12}\big(\frac{1-|x|^2}{2}\big)^{\frac{N-1}{2}p-N}dx <\infty$$
provided $p>2.$
This completes the proof. 
\end{proof}
 
But the above theorem fails for $p=2$ and $2^*$. 
\end{subsection}

\begin{subsection}{Palais Smale Characterisation} In this section we study the Palais Smale sequences of the problem
\begin{equation}{\label{A}}
\left.\begin{array}{rlllll}
-\De_{\Bn}u-\la u &=& |u|^{p-1}u\ & \mbox {in}& \Bn\\
 u &\in& H^1(\Bn) 
\end{array}\right\}
\end{equation}

where $0\le\la<(\frac{N-1}{2})^2$ and $1<p\le 2^\ast-1=\frac{N+2}{N-2}$. To be precise define the associated energy functional $I_{\la}$ as
\begin{eqnarray}
I_{\la}(u)= \frac{1}{2}\Ibn \big[|\na_{\Bn}u|_{\Bn}^2-\la u^2\big]dV_{\Bn}-\frac{1}{p+1}\Ibn |u|^{p+1}dV_{\Bn} 
\end{eqnarray}
We say a sequence $\un\in H^1(\Bn)$ is a Palais Smale sequence ( PS sequence) for $I_{\la}$ 
at a level $d$ if $I_{\la}(\un)\to d$ and $I_{\la}^{'}(\un)\to 0$ in $H^{-1}(\Bn)$.
One can easily see that PS sequences are bounded. Therefore if we restrict $I_{\la} $ to $H^1_r(\Bn)$ and $p< \frac{N+2}{N-2}$ then it follows from Theorem \ref{cpt} that every PS sequence has a convergent subsequence. This is not the case if we relax either one of the above conditions as we will see below. In this section we will analyse this lack of compactness of PS sequences.\\
First observe that the equation \eqref{A} is invariant under isometries. i.e., if $u$ is a solution of \eqref{A} and $\tau \in I(\Bn)$, then $v= u\circ \tau $ is also a solution of \eqref{A}. Thus for a solution $U$ of \eqref{A}, if we define 
\begin{equation}{\label{B}}
 u_n=U\circ \tau_n
 \end{equation}
where $\tau_n \in I(\Bn)$ with $\tau_n(0) \rightarrow \infty$, then $u_n$ is a PS sequence converging weakly to zero. We will see that in the subcritical case noncompact PS sequences are made of finitely many sequnces of type \eqref{B}.\\
However in the critical case $p=2^\ast-1$ we can exhibit another PS sequence coming from the concentration phenomenon.\\
Let $V$ be a solution of the equation
\begin{equation}{\label{C1}}
 -\Delta V \;=\; |V|^{2^\ast-2}V\;,\;\; V\in D^{1,2}(\Rn) 
\end{equation}
The associated energy $J(V)$ is given by \begin{equation}{\label{EEn}}J(V)=\frac{1}{2}\int_{\Rn}|\na V|^2 dx-\frac{1}{2^{*}}\int_{\Rn}|V|^{2^{*}}dx
\end{equation}
Fix $x_0 \in \Bn$ and $\phi \in C_c^\infty (\Bn)$ such that $0\le \phi \le 1$ and $\phi = 1$
in a neighborhood of $x_0$. Define 
\begin{equation}{\label{C}}
v_n = \left(\frac{1-|x|^2}{2}\right)^{\frac{N-2}{2}}\phi(x)  \epsilon_n^{\frac{2-N}{2}}V((x-x_0)/\epsilon_n)
\end{equation}
where $\epsilon_n >0$ and $\epsilon_n \rightarrow 0$, then direct calculation shows that $v_n$ is also a PS sequence. Moreover we have
\begin{lemma} Let $u_n$ be a PS sequence of \eqref{A}, and $\tau_n \in I(\Bn)$ then $v_n := u_n\circ \tau_n$ is also a PS sequence of \eqref{A}.
\end{lemma}

Thus if $\tau_n \in I(\Bn)$ and $v_n$ as in \eqref{C} then $u_n= v_n\circ \tau_n$ is also a PS sequence. We show that any PS sequence is essentially a superposition of the above type of PS sequences.\\
\begin{teor}\label{PS}
 Let ${\un}\in H^1(\Bn)$ be a PS sequence of $I_{\la}$ at a level $d\geq 0$. Then $\exists 
n_1, n_2\in \N$ and functions $u_n^j\in H^1(\Bn)$, $0\leq j\leq n_1$, $v_n^k\in H^1(\Bn)$, 
$0\leq k\leq n_2$ and $u\in H^1(\Bn)$ s.t upto a subsequence
\begin{displaymath}
 \un=u+\sum_{j=1}^{n_1}u_n^j+\sum_{k=1}^{n_2}v_n^k+o(1) 
\end{displaymath}
where $I_{\la}^{'}(u)=0$,  $u_n^j$, $v_n^k$ are PS sequences of the form (\ref{B})
and (\ref{C}) respectively and $o(1)\to 0$ in $\Hunb$. Moreover
\begin{displaymath}
 d= I_{\la}(u)+ \sum_{j=1}^{n_1}I_{\la}(U_j) + \sum_{k=1}^{n_2}J(V_k)+o(1)
\end{displaymath}
where $U_j,V_k$ are the solutions of \eqref{A} and \eqref{C1} corresponding to $u_n^j$,and  $v_n^k$.
\end{teor}

Classifiacation of PS sequences has been done for various problems in bounded domains in $\Rn$
 and on compact Riemannian manifolds, where the lack of compactness is due to the concentration 
phenomenon (See \cite{ST},\cite{SU},\cite{BS},... and the references therein). However the present case should be compared
 with the case of infinite volume case, say the critical equations in $\Rn$.
In this case lack of compactness can occur through vanishing of the mass (in the sense of the concentration compactness of Lions). However in the Euclidean case by dialating a given sequence we can assume that all the functions involved has a fixed positive mass in a given ball and hence we can overcome the vanishing of the mass. However in the case of $\Bn$ this is not possible as the conformal group of $\Bn$ is the same as the isometry group. We will overcome this difficulty by doing a concentration function type argyment near infinity. For this purpose let us define
\begin{defi}For $r>0$, define
$S_r := \{x \in \R^n \;:\; |x|^2=1+r^2 \} $
and for $a \in S_r$ define $$A(a,r)=B(a,r)\cap \Bn$$ where $B(a,r)$ is the open ball in the Euclidean space with center $a$ and radius $r>0.$ 
\end{defi}
Note that for the above choice of $a$ and $r$, $\partial B(a,r)$ is orthogonal to $S^{N-1}.$ We also have,
\begin{lemma} \label{tran} Let $r_1>0, r_2>0$ and $A(a_i,r_i), i=1,2$ be as in the above definition,
 then there exists $\tau \in I(\Bn)$ such that $\tau(A(a_1,r_1)) = A(a_2,r_2)$.
\end{lemma}
\begin{proof} Let $M : \Bn \rightarrow \Hn$ be the standard isometry between the ball model 
$\Bn$ and the upper half space model $\Hn$. Since $M$ is the restriction to $\Bn$ of a 
conformal map of the extended Euclidean space $\Rn \cup \{\infty\}$, we see that the sphere 
$S(a,r)$ is orthogonal to $S^{N-1}$ iff $M (S(a,r))$ is a sphere in $\R^N$ orthogonal to 
$\R^{N-1} \times \{0\}$. Thus $M (A(a_i,r_i))$ is of the form $\R^N_+ \cap B(\tilde a_i,\tilde r_i)$.
 Now the map $\tilde T(x) = \tilde a_2 + \frac{\tilde r_2}{\tilde r_1}(x-\tilde a_1) $ is an 
isometry in $\Hn$ and maps ${\R^N_+ \cap B(\tilde a_1,\tilde r_1)}$ on to ${\R^N_+ \cap B(\tilde a_2,\tilde r_2)}$.
 Hence $T = M^{-1}\circ \tilde T(x)\circ  M$ is in $I(\Bn)$ and maps $A(a_1,r_1)$ to $A(a_2,r_2)$.
 
\end{proof}

{\it Proof of Theorem \ref{PS}.} From standard arguments it follows that any PS sequence is 
bounded in $\Hunb$ .  More precisely $I_{\la}(\un)=d+o(1)$
and $\big<I^{'}_{\la}(\un), \un\big>= o(1)||\un||$,  computing $I_{\la}(\un)-\frac{1}{p+1}\big<I^{'}_{\la}(\un), \un\big>$
we get $$||\un||_{\Hunb}^2\leq C+ o(1)||\un||_{\Hunb}$$ and hence boundedness follows. Thus up to a subsequence we may assume $u_n \rightharpoonup u$ in $\Hunb$.\\\\

{\it Step 1.} In this step we will prove the theorem when $u=0$. \\
{\it Proof.} Since  s$\un$ is a PS sequence we have
$$\Ibn\big[|\na_{\Bn}\un|_{\Bn}^2-\la \un^2\big]dV_{\Bn} = \Ibn |\un|^{p+1}dV_{\Bn}+o(1)$$
Since the square root of LHS is an equivalent norm in $H^1(\Bn)$ and $\un$ does not converge 
strongly to zero we get $$\liminf\limits_{n\to\infty}\Ibn |\un|^{p+1}dV_{\Bn}>\de'>0.$$
Let us fix $\de >0$ such that  $0<2\de<\de'<S_{\la}^\frac{p+1}{p-1}$. Let us define the 
concentration function $Q_n:(0, \infty)\to \R$ as follows.
\begin{displaymath}
 Q_n(r)=\sup_{x\in S_r}\int_{A(x,r)}|\un|^{p+1}dV_{\Bn}.
\end{displaymath}
Now $\displaystyle\lim_{r\to 0} Q_n(r)=0$,and $\displaystyle\lim_{r\to \infty} Q_n(r)>\de$ as for large $r, A(x,r) $ approximates the intersection of $\Bn$ with a half space $\{y\in \Rn\;:\; y\cdot x > 0 \}$.  Therefore 
we can choose a sequence $R_n>0$ and $x_n\in S_{R_n}$ s.t 
\begin{displaymath}
\sup_{x\in S_{R_n}}\int_{A(x,R_n)}|\un|^{p+1}dV_{\Bn}=\int_{A(x_n,R_n)}|\un|^{p+1}dV_{\Bn}=\de.
\end{displaymath}
Fix $x_0 \in S_{\sqrt{3}}$ and using Lemma \ref{tran} choose  $T_n \in I(\Bn)$ such that $A(x_n, R_n)= T_n(A(x_0, \sqrt{3}))$. Now define
\begin{displaymath}
 \vn(x)= \un\circ T_n(x)
\end{displaymath}
Since $ T_n$ is an isometry one can easily see that $\{\vn\}$ is a PS sequence of $\Il$ at the same level as $\un$ and
\begin{equation}\label{mass}
 \int_{A(x_0, \sqrt{3})}|\vn|^{p+1}dV_{\Bn}= \int_{A(x_n,R_n)}|\un|^{p+1}dV_{\Bn}=\de=\sup_{x\in S_{\sqrt{3}}}\int_{A(x, \sqrt{3})}|
\vn|^{p+1}dV_{\Bn}
\end{equation}
and $||\vn||_{H_\la}=||\un||_{H_\la}$. Therefore upto a subsequence we may assume $\vn\deb v$ in $\Hunb$,
$\vn\to  v$ in $L^q_{\rm loc}(\Bn),\ \ 2<q<2^*$ and pointwise. Moreover $v$ solves the equation \eqref{A}. Let us consider the two cases:\\

{\it Case 1:} $ v=0$\\
First we will claim that\\
Claim: For any  $1>r > 2-\sqrt{3}$
\begin{displaymath}
 \int_{\Bn \cap \{|x| \ge r \}}| \vn|^{p+1}dV_{\Bn}=o(1)
\end{displaymath}
To do this let us fix a point $a\in S_{\sqrt 3}$. Let $\phi\in C_c^{\infty}(A(a, \sqrt{3}))$ such that $0\leq\phi\leq 1$ where $A(a,\sqrt{3})=B(a,\sqrt{3})\cap\Bn$ $B(a, \sqrt{3})$ is the Euclidean ball with center $a$ and radius $\sqrt{3}.$ Now 
\begin{displaymath}
< \vn, \psi>_{H_\la}=\Ibn|\vn|^{p-1}\vn\psi dV_{\Bn}+o(1)||\psi||
\end{displaymath}
for every $\psi\in H^1(\Bn)$. Now putting $\psi=\phi^2\vn$ in the above identity we get
\begin{displaymath}
 < \vn, \phi^2 \vn>_{H_\la}=\Ibn| \vn|^{p-1}(\phi\vn)^2 dV_{\Bn}+o(1)
\end{displaymath}
A simple computation gives
\begin{eqnarray*}
 < \vn, \phi^2\vn>_{H_\la}&=&\Ibn \big[|\na_{\Bn}(\phi \vn)|^2_{\Bn}-(\frac{1-|x|^2}{2})^2 \vn^2|\na\phi|^2-\la(\phi \vn)^2\big]dV_{\Bn}
\\
&=& ||\phi\vn||^2_{H_\la}-\int_{\mbox{supp}\  \phi}(\frac{1-|x|^2}{2})^2 \vn^2|\na\phi|^2 dV_{\Bn}
\\
&=& ||\phi \vn||^2_{H_\la}+o(1)
\end{eqnarray*} Thus
\begin{equation}\label{F}
||\phi \vn||_{H_\la}^2= \Ibn| \vn|^{p-1}(\phi \vn)^2 dV_{\Bn}+o(1)
\end{equation}
Now using (\ref{F}), Cauchy-Schwartz and the Poincar\'{e}-Sobolev inequality (\ref{PSE}) we get
\begin{displaymath}
 S_{\la,N,p}(\Ibn|\phi \vn|^{p+1}dV_{\Bn})^\frac{2}{p+1}\leq (\Ibn|\phi \vn|^{p+1}dV_{\Bn})^\frac{2}{p+1}(\int_{A(a,\sqrt{3})}| \vn|^{p+1}dV_{\Bn})^\frac{p-1}{p+1}
\end{displaymath}
Now if $\Ibn|\phi \vn|^{p+1}dV_{\Bn}\not\to 0$ as $n\to\infty$ we get
\begin{displaymath}
 \de^\frac{p-1}{p+1}<S_{\la,N, p}<(\int_{A(a,\sqrt{3})}| \vn|^{p+1}dV_{\Bn})^\frac{p-1}{p+1}<\de^\frac{p-1}{p+1}
\end{displaymath}
which is a contradiction. This implies 
\begin{equation}
\label{I}\Ibn|\phi  \vn|^{p+1}dV_{\Bn}\to 0.
\end{equation} 
Since $a\in S_{\sqrt 3}$ is arbitrary, the claim follws.\\\\
If $1<p<2^\ast-1$ this together with the fact that $v_n \rightarrow 0$ in $L^{p+1}_{loc}(\Bn)$ immediately gives a contradiction to \eqref{mass}.Thus let us assume $p=2^\ast-1$.\\ 
Fix $2-\sqrt {3} < r <R <1$ and choose $\theta \in C_c^\infty(\Bn)$ such that $0\le \theta \le 1,\theta(x) =1 $ for $|x|<r$ and $ \theta(x) =0 $ for $|x|>R.$ \\
Define $ \overline v_n = \theta \vn$, then the above claim shows that $\overline v_n$  is 
again a PS sequence and $ \vn = \overline v_n  +o(1)$ where $o(1)\rightarrow 0$ in $H^1(\Bn).$
 Let us consider a conformal change of the metric, from the hyperbolic to the Eucledean metric.
 Define $\tilde\vn = \left(\frac{2}{1-|x|^2}\right)^{\frac{N-2}{2}}\overline v_n $ then $\tilde\vn \in H^1_0(B(0,R))$ and is a PS sequence for the problem
\begin{equation}\label{CA}
-\Delta w  = a(x)w + |w|^{2^\ast-2}w,\;\; w \in H^1_0(B(0,R))
\end{equation}
where $a$ is a smooth bounded function in $B(0,R) $ given by $a(x) = \frac{4\la - N(N-2)}{(1-|x|^2)^2}.$ Now it follows from the PS sequence characterization of \eqref{CA}
(see for example \cite{ST}) that  
$$ \tilde\vn = \sum_{k=1}^{n_2}w_n^k+o(1)$$
where $w_n^k$ is of the form
$$w_n^k(x) = \phi(x)  \epsilon_n^{\frac{2-N}{2}}V_k((x-x_n^k)/\epsilon_n)$$
where $\epsilon_n >0, |x_n^k| \le 2-\sqrt 3,\epsilon_n \rightarrow 0, x_n^k \rightarrow x^k $ 
as $n\rightarrow \infty$, $V_k$'s are solutions of \eqref{C1} and $\phi \in C_c^\infty(\Bn)$ 
such that $0\le \phi \le 1,\phi(x) =1 $ for $|x|<r$ and $ \phi(x) =0 $ for $|x|>R.$ Moreover 
the associated energy $J_\la (\tilde\vn)$ is given by
$$J_\la (\tilde\vn) = \sum_{k=1}^{n_2}J(V_k)+o(1)$$
where $J_\la$ and $J$ are as in \eqref{JL} and \eqref{EEn}.
Thus $$ \vn = \overline v_n  +o(1) = \left(\frac{2}{1-|x|^2}\right)^{-\frac{N-2}{2}}\tilde\vn +o(1) = \sum_{k=1}^{n_2}v_n^k+o(1)$$
where $v_n^k = \left(\frac{2}{1-|x|^2}\right)^{-\frac{N-2}{2}}w_n^k $.\\
Hence the theorem follows in this case.\\

{\it Case 2:} $ v\not=0$\\
Define,$\wn(x)=v\circ T_n^{-1}(x) $ 
Since $v$ solves \eqref{A}, $\wn$ is a PS sequence. Moreover since $\un\deb 0$ in $\Hunb$, we have $T_n^{-1}(0)\to \infty$ and hence $\wn$ is a PS sequence of the form \eqref{B}. We claim that\\
{\it Claim :}  $\un -\wn$ is a PS sequence of $I_\la$ at level $d-I_\la(v).$\\
Since $\vn \rightharpoonup v$ we have 
$||\na_{\Bn}\vn||_2^2= ||\na_{\Bn} v||_2^2 +  ||\na_{\Bn}(\vn- v)||_2^2 +o(1)$
Also from the Brezis Lieb Lemma, we have $ ||\vn||_2^2= || v||_2^2 +  ||(\vn- v)||_2^2 +o(1)$ and $ ||\vn||_{p+1}^{p+1}= || v||_ {p+1}^{p+1}+  ||(\vn- v)||_ {p+1}^{p+1}+o(1)$. Combining these facts with the invariance of $\Il$ under the action of $I(\Bn)$,  we get
$$\Il(\un-\wn)=\Il(\vn - v ) = \Il(\vn) -\Il( v) + o(1)= \Il(\un) -\Il(\wn) + o(1) $$

Next we show that $\un-\wn$ is a PS sequence of $\Il$.\\
First note that
$\Il'(\un-\wn)(\phi)=\Il'(\vn-v)(\phi_n) $
where $\phi_n = \phi \circ T_n$, $||\phi_n||_{\Hunb}=||\phi||_{\Hunb}$. \\
Now to show $|\Il'(\vn-v)(\phi_n)|=o(||\phi_n||)$ we need to basically prove
\begin{equation}\label{E}
\Ibn\big(|\vn|^{p-1}\vn-|v|^{p-1}v-|\vn-v|^{p-1}(\vn-v)\big)\phi_n dV_{\Bn}=o(||\phi_n||)
\end{equation}
because the linear part follows easily. 
Using the H\"{o}lder inequaliy the L.H.S of (\ref{E}) can be estimated by 
\begin{displaymath}
 |\phi_n|_{L^{2^*}(\Bn)}\left[\Ibn\big|(|\vn|^{p-1}\vn-|v|^{p-1}v-|\vn-v|^{p-1}(\vn-v)\big)|
^{\frac {2n}{n+2}} dV_{\Bn}\right]^{\frac {n+2}{2n}}
\end{displaymath}
standard arguments using Vitali's convergence theorem shows that the term inside bracket is of
 $o(1).$ this proves the claim.\\\\
In view of the above claim if $\un -\wn$ does not converge to zero in $\Hunb$ we can repeat the above procedure for the PS sequence $\un -\wn$ to land in case 1 or case 2. In the first case we are through and in the second case either we will end up with a converging PS sequence or else we will repeat the process. But this process has to stop in finitely many stages as any PS sequence has nonnegative energy and each stage we are reducing the energy by a fixed positive constant. This proves Step 1.\\\\
{\it Step 2:} Let $u_n$ be a PS sequence. Then we know that $u_n$ is bounded and hence going to a subsequence if necessary we may assume that $u_n \rightharpoonup u$ in $\Hunb$, pointwise and in $L^{p+1}_{loc}(\Bn).$ Thus as before we can show that $u_n-u$ is a PS sequence converging weakly to zero, at level $d- \Il(u)$. Now the theorem follows from Step 1.

\end{subsection}

\section{A priori estimates} From the standard elliptic theory we know that the solutions of 
\eqref{problem} are in $C^2(\Bn)$. But we do not have any information on the nature of solution as $x \rightarrow \infty$ (equivalently as $|x|\rightarrow 1$). If $u$ is a positive solution of \eqref{problem}, then $u$ is radial with respect to a point and the exact behaviour of $u(x)$ as $x \rightarrow \infty$ has been obtained in  \cite{MS} by analysing the corresponding ode. In the general case we prove

\begin{teor}{\label{sup_th}}
Let $u$ be a solution of (\ref{problem}) then $u(x)\to 0$ and $|\nab u(x)|^2_{\Bn}\to 0$ as $x\to\infty$ in $\Bn$. In particular $u\in L^{\infty}(\Bn)$.
\end{teor}

\begin{proof}
We will prove the theorem in a few steps. First we will show that $u$ is bounded.\\
\textbf{Step 1}: Let $0<R<2R<\frac{3}{4}<1$ then there exists $q>2^{*}$ and a constant $C>0$ such that $|u\circ\tau|_{L^q(B(0,R))}\leq C$ for all $\tau\in I(\Bn)$.
\\
{\it Proof of step 1:} Since $u\circ\tau $ is also a solution of the same equation for any $\tau\in I(\Bn)$ , we will prove this step by proving a bound on $|u|_{L^q(B(0,R))}$ and observing that the bound remains the same if $u$ is replaced by $u\circ\tau$.\\ Define, $\bar u=u^{+}+1$ and  
\begin{eqnarray*}
\um &=&\bar u \  \ \mbox{if}\ \ u<m\\
&=& 1+m \  \ \mbox{if}\ \ u\geq m 
\end{eqnarray*}
For $\ba>0$ define the test function $v=v_{\ba}$ as $v_\ba=\va^2(\um^{2\ba}\bar u-1)$ where $\va\in C_{0}^{\infty}(\Bn), 0\leq\va\leq 1$,
$\va\equiv 1$ in $B(0,r_{i+1})$, supp$\va\subseteq B(0,r_i)$, $R<r_{i+1}<r_i<2R$ and $|\na\va|\leq\frac{C}{r_i-r_{i+1}}$.\\
Then $0\leq v\in H^1_0(B(0,r_i))$ and hence from (\ref{problem}) we can write
\begin{displaymath}
\Ibn\big<\na_{\Bn}u,\na_{\Bn}v\big>dV_{\Bn}=\la\Ibn uvdV_{\Bn}+\Ibn |u|^{p-1}uvdV_{\Bn}
\end{displaymath}
Now substituting $v$ we get\\
L.H.S= 
\begin{displaymath}
\Ibn \big[\um^{2\ba}\va^2\na u\na\bar u+2\ba\um^{2\ba-1}\bar u\va^2\na u\na\um+2\va(\um^{2\ba}\bar u-1)\na u\na\va\big]
(\frac{2}{1-|x|^2})^{N-2}dx
\end{displaymath}
In the support of 1st integral $\na u=\na\bar u$, and in the support of 2nd integral $\um=\bar u$, $\na\um=\na u$.
Therfore using Cauchy-Schwartz along with the above fact we get\\

\begin{eqnarray}{\label{A'}}
L.H.S &\geq&\frac{1}{2}\Ibn\um^{2\ba}\va^2|\nab\bar u|^{2}_{\Bn}dV_{\Bn}+2\ba\Ibn\um^{2\ba}\va^2|\nab\um|^{2}_{\Bn}dV_{\Bn}
\no\\
&-& 2\Ibn\um^{2\ba}\bar u^2|\nab\va|^{2}_{\Bn}dV_{\Bn}
\end{eqnarray}
So we have
\begin{eqnarray}{\label{B'}}
&\frac{1}{2}&\Ibn\um^{2\ba}\va^2|\nab\bar u|^{2}_{\Bn}dV_{\Bn}+2\ba\Ibn\um^{2\ba}\va^2|\nab\um|^{2}_{\Bn}dV_{\Bn}
\no\\
&\leq&\la\Ibn\va^2(\um^\ba\bar u)^2dV_{\Bn}+\Ibn|u|^{p-1}(\um^\ba\bar u)^2\va^2dV_{\Bn}
\no\\
&+&2\Ibn(\um^\ba\bar u)^2|\nab\va|^2dV_{\Bn}
\end{eqnarray}
Now set $\um^\ba\bar u=w$ then we get
\begin{displaymath}
\va^2|\na w|^2\leq(1+\frac{1}{\var})\ba^2\um^{2\ba}|\na\um|^2\va^2+(1+\var)\um^{2\ba}|\na\bar u|^2\va^2
\end{displaymath}
Now choosing $\var=\frac{\ba}{4+7\ba}$ we get
\begin{eqnarray*}
\frac{1}{4(1+2\ba)}|\na(\va w)|^2 &\leq& 2\ba\um^{2\ba}\va^2|\na\um|^{2}+\um^{2\ba}|\na\bar u|^2\va^2
\no\\
&+& \frac{1}{2(1+2\ba)}|\na\va|^2w^2
\end{eqnarray*}
Now using $\frac{2}{4+7\ba}<\frac{1}{2}$ and (\ref{B'})
\begin{eqnarray}\label{C'}
\frac{1}{4(1+2\ba)}\Ibn|\nab(\va w)|^2dV_{\Bn} &\leq&\int_{B(0,r_i)}[\la+|u|^{p-1}](\va w)^2dV_{\Bn}
\no\\
&+&\frac{C}{(r_i-r_{i+1})^2}\int_{B(0,r_i)}w^2 dV_{\Bn}
\end{eqnarray}
Therefore,
\begin{eqnarray}{\label{D'}}
\frac{1}{4(1+2\ba)}\Ibn|\nab(\va w)|^2dV_{\Bn}&\leq& M(R)\int_{B(0,2R)}|w|^2dV_{\Bn}
\no\\
&+&\Ibn |u|^{p-1}(\va w)^2dV_{\Bn}
\end{eqnarray}
Now for any $K>0$
\begin{eqnarray*}\label{K'}
\Ibn |u|^{p-1}(\va w)^2dV_{\Bn}&\leq& K^{p-1}\Ibn(\va w)^2dV_{\Bn}
\no\\
&+&(\int_{|u|>K}|u|^{2^{*}}dV_{\Bn})^\frac{p-1}{2^{*}}(\Ibn(\va w)^{\frac{2\cdot2^{*}}{2^{*}-p+1}}dV_{\Bn})^\frac{2^{*}-p+1}{2^{*}}
\end{eqnarray*}
Now if $p=2^{*}-1$ then choose $K>0$ large enough so that $$(\int_{|u|>K}|u|^{2^{*}}dV_{\Bn})^\frac{p-1}{2^{*}}<\frac{1}{8(1+2\ba)}$$
and if $p<2^{*}-1$ then choose $K>0$ large enough so that $$(\int_{|u|>K}|u|^{2^{*}}dV_{\Bn})^\frac{p-1}{2^{*}}|V_{\Bn}(B(0,2R))|^{\frac{2}{q}-\frac{2}{2^{*}}}<\frac{1}{8(1+2\ba)}$$
where $q=\frac{2\cdot 2^{*}}{2^{*}-p+1}$. So that in both the cases 
\begin{equation}
\int_{|u|>K} |u|^{p-1}(\va w)^2dV_{\Bn}\leq \frac{1}{8(1+2\ba)}\Ibn|\nab(\va w)|^{2}_{\Bn}dV_{\Bn}
\end{equation} 
Hence from (\ref{D'}) we get
\begin{displaymath}
\Ibn|\nab(\va w)|^2dV_{\Bn}\leq C(R)\int_{B(0,2R)}|w|^2dV_{\Bn} 
\end{displaymath}
i.e, $(\int_{B(0,R)}|w|^{2^{*}}dV_{\Bn})^\frac{2}{2^{*}}\leq C(R)\int_{B(0,2R)}|w|^2dV_{\Bn} $\\
Now notice that, $\um^{\ba+1}\leq w=\um^\ba\bar u\leq\bar u^{\ba+1}$. Now if we choose $2(\ba+1)=2^{*}$,
i.e. $\ba=\frac{2}{N-2}$ then we can write
\begin{displaymath}
(\int_{B(0,R)}\um^{(\ba+1)2^{*}}dV_{\Bn})^\frac{2}{2^*}\leq C(R)\int_{B(0,2R)}\bar u^{2^*}dV_{\Bn}
\end{displaymath}
Now letting $m\to\infty$ we get 
\begin{displaymath}
(\int_{B(0,R)}\bar u^{(\ba+1)2^{*}}dV_{\Bn})^\frac{2}{2^*}\leq C(R)||u||_{\Hunb}^{2^{*}}
\end{displaymath}
Hence $u\in L^{(\ba+1)2^{*}}(B(0,R))$ and $|u|_{L^{(\ba+1)2^{*}}(B(0,R))}\leq C(N, u, ||u||_{\Hunb})$.\\
Now note that dependence of the constant $C$ on $u$ is because of the fact that constant $K$ in
(\ref{K'}) depends on $u$. Now  $$(\int_{|u|>K}|u|^{2^{*}}dV_{\Bn})^\frac{p-1}{2^{*}}<\frac{1}{8(1+2\ba)}$$
implies  $(\int_{|v|>K}|v|^{2^{*}}dV_{\Bn})^\frac{p-1}{2^{*}}<\frac{1}{8(1+2\ba)}$ where $v=u\circ\tau$
for $\tau\in I(\Bn)$. Hence Step 1 follows.
\spa

\textbf{Step 2}: Let $R$ be as in Step 1, then  there exists $C>0$ such that $\displaystyle\sup_{B(0,R)}|u\circ\tau|\leq C$, for all $\tau\in I(\Bn)$ and hence $u$ is bounded.

\spa

{\it Proof:} As in the previous step we will prove $\displaystyle\sup_{B(0,R)}|u|\leq C$ and the constant remains unchanged if $u$ is replaced by $u\circ\tau$.\\ Thanks to the Step 1 we have $\la+|u|^{p-1}\in L^{q}_{B(0,2R)}$ for 
some $q>\frac{N}{2}$. Define, $\la+|u|^{p-1}=g$, hence $|g|_{L^{q}_{B(0,2R)}}\leq C(R,||u||_{\Hunb})$.
From the expression (\ref{C'}) of step 1 we can see that
\begin{eqnarray}{\label{E'}}
\frac{1}{4(1+2\ba)}\Ibn|\nab(\va w)|^2dV_{\Bn} &\leq& C(R,||u||_{\Hunb})|(\va w)^2|_{L^{q^{'}}(\Bn)}
\no\\
&+&\frac{C}{(r_i-r_{i+1})^2}\int_{B(0,r_i)}w^2 dV_{\Bn}
\end{eqnarray}
where $\frac{1}{q}+\frac{1}{q^{'}}=1$. Since $q>\frac{N}{2}\Rightarrow q^{'}<\frac{N}{N-2}=r$(say). Now
let $\frac{1}{q^{'}}=\theta+\frac{1-\theta}{r}$. then using interpolation inequality we get
\begin{displaymath}
|(\va w)^2|_{L^{q^{'}}}\leq\var(1-\theta)|(\va w)^2|_{L^r}+c_1\var^{-\frac{1-\theta}{\theta}}|(\va w)^2|_{L^1}
\end{displaymath}
where $\theta$ depends on $N, t, q^{'}$. Note that $2r=2^{*}$.
Therefore,
\begin{displaymath}
|(\va w)^2|_{L^r}=|\va w|_{L^{2^{*}}(\Bn)}^2\leq C|\nab(\va w)|^2_{L^2(\Bn)}
\end{displaymath}
Hence
\begin{displaymath}
|(\va w)^2|_{L^{q^{'}}}\leq C\var|\nab(\va w)|^2_{L^2(\Bn)} +C\var^{-\al}|(\va w)^2|_{L^1(\Bn)}
\end{displaymath}
Now choosing $\var$ suitably we can write from (\ref{E'})
\begin{displaymath}
\Ibn|\nab(\va w)|^2dV_{\Bn} \leq \frac{C(1+\ba)^\al}{(r_i-r_{i+1})^2}\int_{B(0,r_i)}w^2 dV_{\Bn}
\end{displaymath}
where $C$ depends on $||u||_{\Hunb},N$. Now using Poincar\'{e}-Sobolev inequality in the above expression 
we get
\begin{displaymath}
(\int_{B(0,r_{i+1})}w^{\frac{2N}{N-2}}dV_{\Bn})^\frac{N-2}{N}\leq\frac{C(1+\ba)^\al}{(r_i-r_{i+1})^2}\int_{B(0,r_i)}w^2 dV_{\Bn}
\end{displaymath}
Now using $\chi=\frac{N}{N-2}>1$, $w=\um^\ba\bar u$, $\um\leq\bar u$ and $\ga=2(\ba+1)$ we get
\begin{displaymath}
\big(\int_{B(0,r_{i+1})}\um^{\ga\chi}dV_{\Bn}\big)^\frac{1}{\ga\chi}\leq\big[\frac{C(1+\ba)^\al}{(r_i-r_{i+1})^2}\big]^\frac{1}{\ga}(\int_{B(0,r_i)}\bar u^\ga dV_{\Bn})^\frac{1}{\ga}
\end{displaymath}
Now letting $m\to\infty$ we get
\begin{displaymath}
\big(\int_{B(0,r_{i+1})}\bar u^{\ga\chi}dV_{\Bn}\big)^\frac{1}{\ga\chi}\leq\big[\frac{C(1+\ba)^\al}{(r_i-r_{i+1})^2}\big]^\frac{1}{\ga}(\int_{B(0,r_i)}\bar u^\ga dV_{\Bn})^\frac{1}{\ga}
\end{displaymath}
provided $|\bar u|_{L^\ga(B(0,r_i))}$ is finite. $C$ is a positive constant independent of $\ga$. Now 
we will complete the proof by iterating the above relation. Let us take $\ga=2, 2\chi, 2\chi^2$...
i.e. $\ga_i=2\chi^i$ for $i=0,1,2,$... $r_i=R+\frac{R}{2^i}$. Hence for $\ga=\ga_i$ we get
\begin{displaymath}
\big(\int_{B(0,r_{i+1})}\bar u^{\ga_{i+1}}dV_{\Bn}\big)^\frac{1}{\ga_{i+1}}\leq C^\frac{i}{\chi^i}\big(\int_{B(0,r_{i})}\bar u^{\ga_{i}}dV_{\Bn}\big)^\frac{1}{\ga_{i}}
\end{displaymath}
$C>1$ is a constant depends on $R,N,||u||_{\Hunb}$. Now by iteration we obtain
\begin{displaymath}
\big(\int_{B(0,r_{i+1})}\bar u^{\ga_{i+1}}dV_{\Bn}\big)^\frac{1}{\ga_{i+1}}\leq C^{\Si\frac{i}{\chi^i}}\big(\int_{B(0,2R)}\bar u^2dV_{\Bn}\big)^\frac{1}{2}
\end{displaymath}
Letting $i\to\infty$ we obtain
\begin{displaymath}
\displaystyle\sup_{B(0,R)}\bar u\leq\tilde C|\bar u|_{L^2(B(0,2R))}\leq C_1||u||_{\Hunb}\leq C||u||_{H_{\la}}
\end{displaymath}
Hence $u^{+}$ is bounded in $B(0, R)$. Applying the same argument to $-u$ instead of $u$ we get $u^{-}$
is also bounded by the same. Since we can take $\tau = \tau_b,$ the hyperbolic translation for any $b\in \Bn$ we get $\displaystyle\sup_{\Bn}|u|\leq C$.

\spa

\textbf{Step 3}: $u(x)\to 0$ and $|\nab u(x)|^2_{\Bn}\to 0$ as $x\to\infty$ in $\Bn$.
\spa

{\it Proof:} Let $\bn\in \Bn$ such that  $\bn\to\infty$. Let $ \tau_n\in I(\Bn)$ be the hyperbolic isometry such that  
$\tau_n(0)=\bn$. Define $\vn=u\circ\tau_n$, then we know that $\vn\deb 0$ in $H^1(\Bn)$. Since $\vn$'s are uniformly bounded and $\vn$ satisfies \eqref{problem}, we get $\De_{\Bn}\vn$ is uniformly bounded and hence $\vn$'s are uniformly bounded in $ W^{2,p}_{loc}(\Bn), \ \ 
\fa p, 1<p<\infty$. Combining with Sobolev embedding theorem we get $\vn \to 0$ in $C^1(B(0,\frac{1}{2}))$. In prticular $|\vn(0)|\to 0 $ and $|\na \vn (0)| \to 0 $. Writing in terms of $u$, we get $|u(\bn)|\to 0 $ and $(1-|\bn|^2)|\na u(\bn)| \to 0 .$ Now the theorem follows as $|\nab u(x)|^2_{\Bn}=(\frac{2}{1-|x|^2})^2|\nab u|^2 = (\frac{1-|x|^2}{2})^2|\na u|^2 $.
\end{proof}
Next we prove an improvement of the above result under some restrictions on $\la.$
\begin{teor}\label{cont-2nd}
\label{cont}If u is solution to the equation \eqref{problem} and $\la\leq\frac{N(N-2)}{4}$ 
then $u(x)\leq C(\frac{1-|x|^2}{2})^\frac{N-2}{2}$.
\end{teor}
\begin{proof} First consider the case $p=2^{*}-1$. In this case
using the conformal change of metric we know that if $u$ solves \eqref{problem} then 
$v=(\frac{2}{1-|x|^2})^\frac{N-2}{2}u$ solves \eqref{EE}. Now if $\la\leq\frac{N(N-2)}{4}$ 
then $\tilde\la\leq 0$. Thus, enough to prove\\
{\it Claim:} $v\in L^{\infty}(\Bn)$\\
From standard elliptic theorey we know that $v \in C^2(\Bn).$ Now to prove the bound near infinity, we do a Moser iteration. Fix a point $x_0\in\R^N$ such that $|x_0|=1.$ 
Define $\bar v=v^{+}+1$ and 
\begin{eqnarray*}
\vm &=&\bar v \  \ \mbox{if}\ \ v<m\\
&=& 1+m \  \ \mbox{if}\ \ v\geq m 
\end{eqnarray*}
For $\ba>0$ and $x\in\Bn$ define the test function $w=w_{\ba}$ as $w_\ba=\va^2(\vm^{2\ba}\bar v-1)$ where $\va\in C_{0}^{\infty}(\Rn), 0\leq\va\leq 1$,
$\va\equiv 1$ in $B(x_0,r_{i+1})$, supp$\va\subseteq B(x_0,r_i)$, $R<r_{i+1}<r_i<2R<\frac12 $  and $|\na\va|\leq\frac{C}{r_i-r_{i+1}}$.\\
Then $0\leq w\in H^1_0(\Bn)$ and hence from \eqref{EE} we can write
\begin{displaymath}
\Ibn\big<\na v,\na w\big>dx=\tilde\la\Ibn (\frac{2}{1-|x|^2})^2 vwdx+\Ibn |v|^{2^{*}-1}vwdx
\end{displaymath}
Now substituting $w$ we get\\
L.H.S= 
\begin{displaymath}
\Ibn \big[\vm^{2\ba}\va^2\na v\na\bar v+2\ba\vm^{2\ba-1}\bar v\va^2\na v\na\vm+2\va(\vm^{2\ba}\bar v-1)\na v\na\va\big]dx
\end{displaymath} 
Again as in Step 1 of Theorem \ref{sup_th} we get
\begin{eqnarray}
&\frac{1}{2}&\Ibn\vm^{2\ba}\va^2|\na\bar v|^{2}dx+2\ba\Ibn\vm^{2\ba}\va^2|\na\vm|^{2}dx
\no\\
&\leq&\tilde\la\Ibn(\frac{2}{1-|x|^2})^2\va^2(\vm^\ba\bar v)^2dx+\Ibn|v|^{2^{*}-2}(\vm^\ba\bar v)^2\va^2dx
\no\\
&+&2\Ibn(\vm^\ba\bar v)^2|\na\va|^2dx
\end{eqnarray}
Since $\tilde\la\leq 0$ we can ignore the term which contains singularity at the origin 
to obtain 
\begin{eqnarray}
\frac{1}{2} \Ibn\vm^{2\ba}\va^2|\na\bar v|^{2}dx+2\ba\Ibn\vm^{2\ba}\va^2|\na\vm|^{2}dx \no\\
\leq \Ibn|v|^{2^{*}-2}(\vm^\ba\bar v)^2\va^2dx + 2\Ibn(\vm^\ba\bar v)^2|\na\va|^2dx \no
\end{eqnarray}
Now we can do the standard Moser iteration techniques as in Step 1 and 
Step 2 of Theorem \ref{sup_th} to conclude $v\in L^{\infty}(B(x_0, R)\cap \Bn)$. Since $x_0$ is arbitrary  and we can cover $\Bn \cap \{x : |x|\ge \frac R2 \}$ by finitely many sets of the form $B(x_0, R)\cap \Bn$ the claim follows.\\

When $p< 2^\ast -1$, the conformal change will give us an equation of the form 
$$-\De v - \tilde\la \left(\frac{2}{1-|x|^2}\right)^2v =  \frac{|v|^{p_t-1}v}{(1-|x|^2)^t}, v \in H^1_0(\Bn)$$ where $ t= N-\frac{N-2}{2}(p+1)$. Again one can proceed as before to do a Moser iteration to get the result. Of course while estimating the terms on the RHS one has to use the Hardy inequality $\int\limits_{\Bn}|\na u|^2 dx \ge C \left(\int\limits_{\Bn} \frac{|u|^{p_t}}{(1-|x|^2)^t} dx \right)^{\frac{2}{p_t}} $ in place of the usual Sobolev inequality.
\end{proof}

\section{Existence and Non Existence of sign changing radial solutions}

In this section we will study the existence and non existence of sign changing solutions of the problem

\begin{equation}
\label{H}
\left.\begin{array}{rlllll}
 -\De_{\Bn} u-\la u &=& |u|^{p-1}u\ \ \mbox {in}\ \Bn\\
u &\in& H^1(\Bn)\\
\end{array}\right\}
\end{equation}
where $\la < (\frac{N-1}{2})^2 $ and $1<p\le\frac{N+2}{N-2}$.\\
We will see below that there is a significant difference between the cases $1<p<\frac{N+2}{N-2}$ and $p=\frac{N+2}{N-2}$. In the subcritical case we have

\begin{teor} {\label{K}} Let $1<p<\frac{N+2}{N-2}$, then there exists a sequence of solutions $u_k$ of \eqref{H} such that $||u_k||\rightarrow \infty $ as ${k\to\infty}$.
\end{teor}
\begin{rmk}The above result holds when $\la= (\frac{N-1}{2})^2$, with $u_k \in \cal{H}$ and the corresponding norm goes to infinity as $k\rightarrow \infty.$
 \end{rmk}

As an immediate corollary we obtain the existence of sign changing solutions for the Hardy-Sobolev-Mazya equation and the critical Grushin equation.\\
\begin{teor}\label{HSMSC} The Hardy-Sobolev-Mazya equation \eqref{P} admits a sequence $v_k$ of sign changing solutions such that $||\na v_k||_2 \rightarrow \infty$ as ${k\to\infty}$.
\end{teor}
\begin{proof} As mentioned in the introduction cylindrically symmetric solutions of \eqref{P}
are in one one correspondence with the solutions of \eqref{problem} with $N=n-k+1$ and $p=p_t$ . One can easily see that  $p=p_t < \frac{N+2}{N-2}$, thus Theorem \ref{K} apply . Let $v_k$ be the solution of \eqref{P} corresponding to $u_k$, then since $ ||u_k|| \rightarrow \infty$ we get $||\na v_k||_2 \rightarrow \infty$ (see \cite{MS},section 6, for details). 
\end{proof}
Similarly we have 
\begin{teor}\label{GRUSC} The critical Grushin equation \eqref{GRU} admits a sequence $v_k$ of sign changing solutions such that $\|| v_k\||_2 \rightarrow \infty$ as ${k\to\infty}$.
\end{teor}

However in the critical case \eqref{H} does not have a solution always. In fact it 
follows from the results in \cite{S} that the Dirichlet problem
$$-\De_{\Bn} u-\la u = u^{\frac{N+2}{N-2}}\ \ \mbox {in}\ \ \ \Omega,\; u>0 \ \ \mbox {in}\ \ \ \Omega, u=0 \ \ \mbox {on}\ \ \ \partial\Omega $$
does not have a solution in a bounded star shaped domain $\Ome\subset\Bn$ if $N=3$ or $\la\leq\frac{N(N-2)}{4}$.
Non existence of positive solution to \eqref{H} in the case of $N=3$ or $\la\leq\frac{N(N-2)}{4}$ was established
in (\cite{MS}). Since the existence of a nontrivial sign changing radial solution to \eqref{H} gives a solution to the above problem in a 
geodesic ball we conclude:

\begin{teor}\label{nonex} Let $p=\frac{N+2}{N-2}$, then
the equation \eqref{H} does not have a radial sign changing solution if  $\la\leq\frac{N(N-2)}{4}$ or $N=3$.
\end{teor}

\begin{teor}{\label{CP}}Let $p=\frac{N+2}{N-2}$, then the Equation \eqref{H} has at least
two pairs of non-trivial radial solutions if 
 $N\geq 7$ and $\frac{N(N-2)}{4}<\la<(\frac{N-1}{2})^2$.
\end{teor}

Solutions of \eqref{H} is in one to one correspondence with the critical points of the 
functional
\begin{displaymath}
I(\la)(u)=\frac{1}{2}\Ibn(|\na_{\Bn}u|^2-\la u^2)dV_{\Bn}-\frac{1}{p+1}\int_{\Bn}|u|^{p+1} dV_{\Bn}, u \in H^1(\Bn).
\end{displaymath}
From the Poincar\'{e}-Sobolev inequality we know that $J$ is well defined and $C^1$ in $H^1(\Bn)
$. The main difficulty in finding critical points of $J$ is due to the lack of compactness, which we have already analysed in the last section.\\\\
{\it {\bf Proof of Theorem \ref{K}.}} Thanks to Theorem $J:H^1_r(\Bn)\rightarrow \R $ satisfies the Palais-Smale condition and hence using the standard arguments using genus as in  Ambrosetti-Rabinowitz ( \cite{AR} Theorems 3.13 , 3.14 ) we get a sequnce $u_k,k=1,2,...$ of critical points for $J|_{H^1_r(\Bn) }$ with $||u_k||\rightarrow \infty $. Also we know that the critical points of $J|_{H^1_r(\Bn) }$ are critical points of $J$ in $H^1(\Bn)
$ (see \cite{P}) as well. this proves the theorem.\\

{\bf Proof of Theorem \ref{CP}.} We know from \cite{MS} that \eqref{H} has a unique positive raidal solution say $u_0$. In order to prove the existence of a sign changing solution we proceed as in \cite{CSS} (see \cite{V} for the same kind of result on compact Riemannian manifolds).\\
First recall that the unique positive radial solution $u_0$ satisfies
$$ \frac{1}{N}S_{\la}^\frac{N}{2}=\Il(u_0) = \inf\limits_{u\in N} \Il(u)$$
where $\mathcal{N}$ is the Nehari manifold
$$\mathcal{N}= \left\{u \in H^1(\Bn)\setminus \{0\} \;:\; \Ibn(|\na_{\Bn}u|^2-\la u^2)dV_{\Bn} = \int_{\Bn}|u|^{p+1} dV_{\Bn} \right\} $$
Next observe that if $u$ is a sign changing solution then $u^\pm \in \mathcal{N}.$ Thus to look for a sign 
changing radial solution we need to look at only the $H_r^1(\Bn)$ function whose positive and 
negative parts are in $\mathcal{N}$. More precisely let us define for $u \in H_r^1(\Bn)\setminus \{0\}$
 $$f_{\la}(u)=\frac{\int_{\Bn}|u|^{2^*} dV_{\Bn}}{\Ibn(|\na_{\Bn}u|_{\Bn}^2-\la u^2)dV_{\Bn}}$$
and $f_{\la}(0)=0$. Let $\mathcal{N}_1$ and $U$ be defines as
 $$\mathcal{N}_1=\{u\in H_r^1(\Bn):f_{\la}(u^{+})=f_{\la}(u^{-})=1\}$$ 
$$U=\{u\in H_r^1(\Bn):|f_{\la}(u^{\pm})-1|<\frac{1}{2}\}.$$
We can easily check that $U\not=\emptyset$ and the Poincare Sobolev inequality tells us that there exists $\al>0$ such that $u\in U\Rightarrow||u^{\pm}||>\al$.\\
{\it Claim: Let $\beta = \inf\limits_{u\in \mathcal{N}_1} I_{\la}(u) $ then there exists a PS sequence $\{u_n\} $ of $I_\la$ at the level $\beta$ such that $u_n \in U$ for all $n$. Moreover $\beta$ satisfies the estimate $$\beta <\frac{S_\la^{\frac{N}{2}}}{N} + \frac{S^{\frac{N}{2}}}{N} .$$ }

Assuming the claim, let us observe from Theorem (\ref{PS}) that the PS sequence otained in the
above claim must be of the form $u_n= u + o(1)$ where $u$ is a nontrivial solution of \eqref{H}. Since $I_\la(u)=\beta$ we immediately see that $u$ changes sign and hence the theorem follows. Now it remains to prove the claim.\\
{\it Proof of claim:}  Existence of the PS sequence at level $\beta$ follows exactly as in 
\cite{CSS}. We will just outline the arguments and refer to \cite{CSS} and the references 
therein for details.\\
Let us define $P$ to be the cone of non negative functions in $H^1_r(\Bn)$ and $\varSigma$ be the collection of maps $\si\in C(Q,H^1_r(\Bn))$ where $Q=[0, 1]\times [0, 1]$, satisfying
$$\si(s, 0)=0,\;
\si(0, t)\in P,\;\si(1, t)\in -P,\;(\Il\circ\si)(s, 1)\leq 0,\;f_{\la}\circ\si(s, 1)\geq 2$$
for all $s, t \in [0, 1]$. 
The very same arguments used in \cite{CSS} tells us that 
$$\beta =\displaystyle\inf_{\si\in\varSigma} \sup_{u\in\si(Q)}\Il(u).$$
If $\beta$ is not a critical level then we can use a variant of the standard deformation lemma to conclude that the above min max level can be further lowered leading to a contradiction (See \cite{CSS} for details).
 Thus the crucial step to prove is the estimate on $\beta.$
Let $\phi \in C_c^\infty(\Bn)$ such that $0\le \phi \le 1 $ and $\phi =1$ on $|x|<r$ where $0<r<1.$ Define $v_\var $ as
$$ v_\var (x) \;=\; \phi(x) \left(\frac{\var}{\var^2+|x|^2}\right)^{\frac{N-2}{2}}$$
Define 
$$ u_\var (x) \;=\; (\frac{1-|x|^2}{2})^{\frac{N-2}{2}}v_\var (x)$$
Let $u_0$ be the unique positive solution of \eqref{H} then for suitable $a,b \in \R,\;t[(1-s)au_0 +bsu_\var] \in \varSigma$. Thus the estimate on $\beta$ follows once we show that 
$$\sup\limits_{a,b \in \R}\Il(au_0 +bu_\var) <\frac 1N \left(S_\la^{\frac{N}{2}} + S^{\frac{N}{2}}\right) .$$
Making a conformal change enough to show that 
$$\sup\limits_{a,b \in \R}J_{\la}(av_0 +bv_\var) <\frac 1N \left(S_\la^{\frac{N}{2}} + S^{\frac{N}{2}}\right) .$$
where $J_{\la}$ is as in (\ref{JL}) and $v_0 (x)= (\frac{2}{1-|x|^2})^{\frac{N-2}{2}}u_0 (x).$\\
Before proceeding to prove this we need to calculate few estimates. \\
Setting, $\var^2=\mu$ we find $v_{\var}=\mu^{\frac{N-2}{4}}\frac{\va(x)}{(\mu+|x|^2)^\frac{N-2}{2}}=: \mu^{\frac{N-2}{4}}w_{\mu}$\\
Now let us recall some results from \cite{BN} 
\begin{itemize}
\item[(i)] $|\na w_{\mu}|_{L^2(\Bn)}^2=\frac{S^{\frac{N}{2}}}{\mu^{\frac{N-2}{2}}}+O(1)$ 
\item[(ii)] $|w_{\mu}|_{L^2(\Bn)}^{2}=\frac{C_1}{\mu^{\frac{N-4}{2}}}+O(1)$
\item[(iii)] $|w_{\mu}|_{L^{2^{*}}(\Bn)}^{2^{*}}=\frac{S^\frac{N}{2}}{\mu^\frac{N}{2}}+O(1)$
\end{itemize}
Now using (i), (ii), (iii) and the fact that $v_{\var}$ has support in $B(0,R)$ where $R<1$ we can
compute the following estimates
\begin{enumerate}
\item $|\na v_{\var}|_{L^2(\Bn)}^2=S^\frac{N}{2}+O(\mu^\frac{N-2}{2})$
\item $|(\frac{2}{1-|x|^2})v_{\var}|_{L^2{\Bn}}^2=C_2\mu+O(\mu^\frac{N-2}{2})$
\item $|v_{\var}|_{L^{2^{*}}(\Bn)}^{2^{*}}=S^\frac{N}{2}+O(\mu^{\frac{N}{2}})$
\item $|v_{\var}|_{L^1(\Bn)}\leq C_3 \mu^\frac{N-2}{4}$
\item $|v_{\var}|_{L^{2^{*}-1}(\Bn)}^{2^{*}-1}\leq C_4 \mu^\frac{N-2}{4}$
\end{enumerate}
For proof of (4) and (5) see appendix. Taking $b=0, a=1$ we can see that $\displaystyle\sup_{a,b \in \R}J_{\la}(av_0 +bv_\var)>0$
and $J_{\la}\big(t(av_0 +bv_\var)\big)<0$ while $|t|\to\infty$ and $a, b$ to be fixed. Therefore
it is enough to consider $\displaystyle\sup_{|a|,|b|<K}J_{\la}(av_0 +bv_\var)$ where $K$ is chosen sufficiently large.
Therefore using the above estimates and $I_{\la}(u_0)=J_\la(v_0)$ we get\\
\begin{eqnarray*}
J_{\la}(av_0 +bv_\var)&=&\frac{1}{2}\Ibn\big[|\na(av_0+bv_\var)|^2-\tilde\la(\frac{2}{1-|x|^2})^2(av_0+bv_{\var})^2\big]dx
\\
&-&\frac{1}{2^*}\Ibn\big[|av_0+bv_\var|^{2^*}dx
\\
&\leq&\frac{a^2}{2}\Ibn(|\na v_0|^2-\tilde\la(\frac{2}{1-|x|^2})^2 v_0^2)dx
\\
&+&\frac{b^2}{2}\Ibn(|\na v_{\var}|^2-\tilde\la(\frac{2}{1-|x|^2})^2 v_{\var}^2)dx
\\
&+&\Ibn\big[\na(a v_0)\cdot \na(bv_{\var})-\tilde\la(\frac{2}{1-|x|^2})^2 av_0 bv_{\var}\big]dx
\\
&-&\frac{a^{2^*}}{2^*}\Ibn|v_0|^{2^*}dx-\frac{b^{2^*}}{2^*}\Ibn|v_{\var}|^{2^*}dx
\\
&+&K_1\big[|bv_{\var}|_{L^{2^*-1}(\Bn)}^{2^*-1}|av_0|_{L^{\infty}}+|bv_{\var}|_{L^1}|av_0|_{L^{\infty}}^{2^*-1}\big]
\\
&\leq&\frac{1}{N}(S_{\la}^\frac{N}{2}+S^\frac{N}{2})-C_5\mu
\\
&+&K_2\big[|bv_{\var}|_{L^1}|\De(av_0)|_{L^{\infty}}+|\tilde\la(\frac{2}{1-|x|^2})^2 bv_{\var}|_{L^1}|av_0|_{L^{\infty}} \big]
\\
&+&K_1\big[|bv_{\var}|_{L^{2^*-1}(\Bn)}^{2^*-1}|av_0|_{L^{\infty}}+|bv_{\var}|_{L^1}|av_0|_{L^{\infty}}^{2^*-1}\big]
\\
&\leq&\frac{1}{N}(S_{\la}^\frac{N}{2}+S^\frac{N}{2})-C_5\mu+K_3|v_0|_{L^{\infty}(B(0,R))}\mu^{\frac{N-2}{4}}
\end{eqnarray*}
Now taking $\var$ to be small enough we can conclude $J_{\la}(av_0 +bv_\var)<\frac{1}{N}(S_{\la}^\frac{N}{2}+S^\frac{N}{2})$
since we have $N\geq 7$ and $\mu=\var^2$.

\section{Appendix} 

In this appendix we will recall a few facts about the Hyerbolic space, especially the disc model. For proofs of theorems and a detailed discussion we refer to \cite{JR}.\\
{\bf Disc and Upper half space model.} We have already introduced the Disc model.
The half space model is given by $(\H^N,g)$ where $\H^N = \{(x_1,...x_N) \in \R^N : x_N>0 \}$
and $g_{ij}= \frac{1}{x_N^2}\delta_{ij}$.\\
The map $M :\H^n \rightarrow \Bn$ given by
\begin{equation}\label{H-B}
 M(x) := \left(\frac {1-r^2 - |z|^2}{(1+r)^2 + |z|^2} , \frac {2 z}{(1+r)^2 + |z|^2}\right) 
\end{equation}
where $x\in \H^N$ is denoted as $(z, r)\in \R^{N-1}\times (0,\infty),$ is an Isometry between $\H^N$ and $\Bn$ with $M^{-1}=M.$\\
{\bf The Hyperbolic distance in $\Bn$.}
The Hyperbolic distance between $x,y \in\Bn$ is given by 
\begin{displaymath}
 d_{\Bn}(x,y)= cosh^{-1}\big(1+\frac{2|x-y|^2}{(1-|x|^2)(1-|y|^2)}\big)
\end{displaymath}
We define the hyperbolic sphere of $\Bn$ with center $b$ and radius $r>0$, as the set 
\begin{displaymath}
 S_{\Bn}(b,r)= \{x\in\Bn: d_{\Bn}(b,x)=1 \}
\end{displaymath}
It easily follows that
a subset $S$ of $\Bn$ is a hyperbolic sphere of $\Bn$ iff $S$ is a Euclidean sphere of $\R^N$ and contained in $\Bn$ probably with a different center and different radius.\\
{\bf Isometry group of $\Bn$.}
Let $a$ be the unit vector in $\Rn$ and $t$ be a real number. Let $P(a,t)$ be the hyperplane $P(a,t)=\{x\in\Rn: x.a=t\}$.The reflection $\rho$ of $\Rn$ in the hyperplane $P(a,t)$ 
 is defined by the formula $\rho(x)=x+2(t-x.a)a$.\\
Now let $b\in\Rn$ and $r$ is positive real number, then the reflection $\sigma$ of $\Rn$ in a sphere $S(b,r)=\{x\in\Rn: |x-b|=r\}$ is defined by the formula $\sigma(x)=b+(\frac{r}{|x-b|})^2(x-b)$.\\
Let us denote the extended Euclidean space by, $\hat{\Rn}:=\Rn\cup\infty$.\\
\begin{defn}
 A sphere $\Sigma$ of $\hat{\Rn}$ is defined to be either a Euclidean sphere $S(a,r)$ or an extended plane $\hat P(a,t)=P(a,t)\cup\{\infty\}$.
\end{defn}
\begin{lemma}{\label{ortho}}
Two spheres of $\hat{\Rn}$ are orthogonal under the following conditions:
\begin{itemize}
\item The spheres $\hat P(a,r)$ and $\hat P(b,s)$ are orthogonal iff $a$ and $b$ are orthogonal.
\item The spheres $S(a,r)$ and $\hat P(b,s)$ are orthogonal iff $a$ is in $\hat P(b,s)$.
\item The spheres $S(a,r)$ and $S(a,r)$ are orthogonal iff $|a-b|^2 = r^2+s^2$.
\end{itemize}
\end{lemma}
For a proof see \cite{JR}, Theorem 4.4.2 .\\
With these definitions we have the following characterisation of the isometry group of $\Bn$.
Again we refer to \cite{JR} for a proof.\\
\begin{teor}{\label{mob}}
Every element of $I(\Bn)$ is a finite composition of reflections of $\hat{\Rn}$ in spheres orthogonal to $S^{N-1}$.
\end{teor}
{\bf Hyperbolic Translation.}
Let $S(a,r)$ be a sphere in $\Rn$ with $r^2= |a|^2-1$, therefore $S(a,r)$ is orthogonal to $S^{N-1}$. Let $\sigma_a$ be the reflection of $\Rn$ in $S(a,r)$ and $\rho_a$ is the reflection of $\Rn$ in the hyperplane $a.x=0$.
Using Lemma \ref{ortho} and Theorem\ref{mob} we get $\sigma_a\circ\rho_a$ is an isometry of $\Bn$ . Define $a^*=\frac{a}{|a|^2}$. Then we get
\begin{displaymath}
 \sigma_a\circ\rho_a(x)=\frac{(|a|^2-1)x+(|x|^2+2x\cdot a^*+1)a}{|x+a|^2}
\end{displaymath}
In particular, $\sigma_a\circ\rho_a(0)=a^*$.\\
Let $b$ be any nonzero point in $\Bn$, and $b^*=a(say)$. Then $|a|>1$ and $a^*=b$. Now if we take $r=(|a|^2-1)^\frac{1}{2}$, then $S(a,r)$ is 
orthogonal to $S^{N-1}$ by Lemma \ref{ortho}. Therefore we can define a M\"{o}bius transformation of $\Bn$ by the formula
\begin{displaymath}
 \tau_b=\sigma_{b^*}\circ\rho_{b^*}.
\end{displaymath}
i.e. \begin{displaymath}
      \tau_b(x)=\frac{(|b^*|^2-1)x+(|x|^2+2x\cdot b+1)b^*}{|x+b^*|^2}
     \end{displaymath}
Therefore in terms of $b$ we can write
\begin{displaymath}
 \tau_b(x)=\frac{(1-|b|^2)x+(|x|^2+2x\cdot b+1)b}{|b|^2|x|^2+2x\cdot b+1}
\end{displaymath}
As $\tau_b$ is the composition of two reflections in hyperplanes orthogonal to the line $(-\frac{b}{|b|}, \frac{b}{|b|})$, the 
transformation $\tau_b$ acts as a translation  along this line. We define $\tau_0$ to be the identity. 
Then $\tau_b(0)=b$ for all $b\in\Bn$. The map $\tau_b$ is called the hyperbolic translation of $\Bn$ by $b$.

\begin{lemma}
(i) $|v_{\var}|_{L^1(\Bn)}\leq C_3\mu^\frac{N-2}{4}$
(ii) $|v_{\var}|_{L^{2^{*}-1}(\Bn)}^{2^{*}-1}\leq C_4\mu^\frac{N-2}{4}$
\end{lemma}
\begin{proof}
(i) $|v_\var|_{L^1}=\mu^\frac{N-2}{4}|w_\mu|_{L^1}$. now using $\va-1\equiv0$ near $0$ we get
\begin{eqnarray*}
\Ibn|w_{\mu}|&=&\Ibn\frac{\va(x)-1}{(\mu+|x|^2)^\frac{N-2}{2}}dx+\Ibn\frac{dx}{(\mu+|x|^2)^\frac{N-2}{2}}  
\\
&=&O(1)+C\mu \  \ \mbox{where} \ \ C=\Ibn\frac{dx}{(1+|x|^2)^\frac{N-2}{2}}
\end{eqnarray*}
Hence $|v_\var|_{L^1}=O(\mu^\frac{N-2}{4})+C\mu^\frac{N+2}{4}\leq C_3\mu^\frac{N-2}{4}$

\spa

(ii)$|v_{\var}|_{L^{2^{*}-1}(\Bn)}^{2^{*}-1}=\mu^\frac{N+2}{4}|w_\mu|_{L^{2^{*}-1}}^{2^{*}-1}$. now
\begin{eqnarray*}
\Ibn|w_{\mu}|^{2^{*}-1}&=&\Ibn\frac{{\va(x)}^{2^{*}-1}-1}{(\mu+|x|^2)^\frac{N+2}{2}}dx+\Ibn\frac{dx}{(\mu+|x|^2)^\frac{N+2}{2}}  
\\
&=&O(1)+C\mu^{-1} \  \ \mbox{where} \ \ C=\Ibn\frac{dx}{(1+|x|^2)^\frac{N+2}{2}}
\end{eqnarray*}
Hence $|v_\var|_{L^{2^{*}-1}(\Bn)}^{2^{*}-1}=O(\mu^\frac{N+2}{4})+C\mu^\frac{N-2}{4}\leq C_4\mu^\frac{N-2}{4}$
\end{proof}

\end{document}